\documentclass[11pt,oneside]{amsart}

\usepackage{amsfonts,amsmath,amssymb,amsthm}
\usepackage{geometry}
\usepackage{graphicx,xcolor}
\usepackage[hidelinks]{hyperref}

\usepackage[alphabetic,abbrev]{amsrefs} 
\renewcommand{\MR}[1]{}

\usepackage{zref-clever} \newcommand\Cref[1]{\zcref[S]{#1}}

\newcommand\newtheoremx[2]{%
\AddToHook{env/#1/begin}{
\zcsetup{countertype={theorem=#1}}}
\newtheorem{#1}[theorem]{#2}
}

\newtheorem{theorem}{Theorem}[section]
\newtheoremx{claim}{Claim}
\newtheoremx{proposition}{Proposition}
\newtheoremx{corollary}{Corollary}
\newtheoremx{lemma}{Lemma}
\newtheoremx{conjecture}{Conjecture}

\newtheorem{introtheorem}{Theorem}

\zcsetup{countertype={introtheorem=theorem}}

\theoremstyle{definition}
\newtheoremx{definition}{Definition}
\newtheoremx{notation}{Notation}
\newtheoremx{assumption}{Assumption}

\theoremstyle{remark}
\newtheoremx{example}{Example}
\newtheoremx{remark}{Remark}
\newtheoremx{question}{Question}

\title{An abelian envelope without the quotient property}

\author{Johannes Flake}
\address{Mathematical Institute, University of Bonn, Bonn, Germany}
\email{flake@math.uni-bonn.de}

\author{Jonathan Gruber}
\address{Chair of Algebra and Representation Theory, RWTH Aachen University, Aachen, Germany}
\email{gruber@art.rwth-aachen.de}

\author{Thorsten Heidersdorf}
\address{School of Mathematics, Statistics and Physics, Newcastle University, Newcastle, United Kingdom}
\email{thorsten.heidersdorf@newcastle.ac.uk}

\date{\today}

\newcommand\kk{\Bbbk}
\newcommand\EE{\mathbb E}

\newcommand\cA{\mathcal A}
\newcommand\cB{\mathcal B}
\newcommand\cC{\mathcal C}
\newcommand\cD{\mathcal D}
\newcommand\cM{\mathcal M}
\newcommand\cP{\mathcal P}
\newcommand\cQ{\mathcal Q}
\newcommand\cR{\mathcal R}
\newcommand\cT{\mathcal T}
\newcommand\cU{\mathcal U}
\newcommand\ZZ{\mathbb Z}

\newcommand\id{\operatorname{id}}
\newcommand\ev{\operatorname{ev}}

\newcommand\co{\operatorname{co}}
\newcommand\im{\operatorname{im}}
\newcommand\Fun{\operatorname{Fun}}
\newcommand\PSh{\operatorname{PSh}}
\newcommand\Vect{\operatorname{Vec}}
\newcommand\Comod{\operatorname{Comod}}
\newcommand\Hom{\operatorname{Hom}}
\newcommand\Tilt{\operatorname{Tilt}}

\newcommand\tforall{\text{for all }}
\renewcommand\o\otimes
\renewcommand\:\colon
\newcommand\op{^{\mathrm{op}}}
\newcommand\fd{^{\mathrm{fd}}}
\newcommand\one{\mathbf 1}
\newcommand\tto{\twoheadrightarrow}

\begin{document}

\begin{abstract}
We show that the universal rigid monoidal category on one object has an abelian envelope.
Since it was proven by Coulembier--Etingof--Ostrik--Pauwels \cite{CEOP} that this category cannot have an abelian envelope with the quotient property, this disproves the conjecture in [op.cit.] saying that every abelian envelope has the quotient property. Monoidal Ringel duality \cite{FG} yields a candidate envelope as a lower finite highest weight category.
Our proof that this is an abelian envelope relies on Coulembier--Etingof's \cite{CE} continuants to show the required universal property.
\end{abstract}

\maketitle

\section{Introduction}

Tensor categories are an axiomatically defined class of categories modeled after categories of representations of finite or algebraic groups. In particular, they lie in the intersection of the classes of abelian categories and rigid monoidal categories. Replacing abelian by additive and idempotent complete (i.e., \emph{pseudo-abelian}) categories leads to the notion of \emph{pseudo-tensor categories}, which encompass additionally categories of projective representations (after adjoining the tensor unit) for finite groups, categories of tilting representations for reductive groups, and many diagrammatic monoidal categories, for instance. An abelian envelope of such a category $\cA$ is a fully faithful monoidal embedding into a tensor category $\cB$ with the following universal property: faithful linear monoidal functors from $\cA$ to tensor categories extend uniquely to exact linear monoidal functors from $\cB$.

The theory of abelian envelopes is at this point still being developed, see \cites{Cou,BEO,CEAH,CEOP,Cou-hom,FLP-abenv}. Constructing an abelian envelope for a given pseudo-tensor category can be a difficult problem and it is known that not every pseudo-tensor category admits an abelian envelope. In \cite{CEOP}, a characterization of abelian envelopes with a \emph{quotient property} is obtained. These are abelian envelopes in which every object is a quotient of an object of the pseudo-tensor subcategory. It is conjectured \cite{CEOP}*{Conjecture~2.2.2} that every abelian envelope is of this kind. We show that this is not the case.

By \cite{CEOP}*{Example~3.2.6}, the universal rigid monoidal linear category $\cP$ on one object, or equivalently, its pseudo-abelian completion $\cA$, which can be regarded as the universal pseudo-tensor category on one object, cannot have an abelian envelope with the quotient property. We show:

\begin{introtheorem} \label{intro-a} $\cA$ has an abelian envelope $\cB$. Moreover, $\cB$ is a lower finite highest weight category in the sense of \cite{BS}, and $\cA$ can be identified with its category of tilting objects.
\end{introtheorem}

Our proof relies on the theory of \emph{continuants} as developed in \cite{CE}, combined with the theory of monoidal Ringel duality developed in \cites{SS,FG}. We present a general sufficient criterion to detect abelian envelopes, roughly as follows:

\begin{introtheorem}[\Cref{thm:main-general}] \label{intro-b} Let $\cB$ be any lower finite highest weight category that is also a tensor category. Assume that its tilting objects form a pseudo-tensor subcategory $\cA$. Then $\cB$ is the abelian envelope of $\cA$ if the tilting objects together with the cohomologies of their unbounded continuants generate all standard and costandard objects under isomorphism, tensor products, duals,  direct sums and summands, and extensions.
\end{introtheorem}

The proof of the existence of the abelian envelope has two distinct parts. We first use the description of morphisms by planar diagrams to show that $\cP$ admits a triangular structure in the sense of \cite{SS}. In conjunction with monoidal Ringel duality \cite{FG} we obtain as in \Cref{intro-a} a lower finite highest weight category $\cB$ whose tilting objects agree with $\cA$. The second part of the proof consists then in identifying certain standard objects with continuant cohomologies and verifying all the required properties of \Cref{intro-b}. In the process, we obtain descriptions of all co-/standard objects in $\cB$. We also identify $\cB$ as the category of comodules of the free Hopf algebra with invertible antipode on the $2\times 2$ matrix coalgebra, see \Cref{rem:quantum-group}.

Our results establish a new connection between the theory of abelian envelopes and that of continuants. They can be regarded as a first step towards a theory of abelian envelopes beyond the quotient property. The following questions arise naturally.
\begin{question}
(a) Are there \emph{symmetric} (or \emph{braided}, or \emph{pivotal}) pseudo-tensor categories with abelian envelopes without the quotient property?

(b) Is the sufficient condition in \Cref{intro-b} for $\cB$ being an abelian envelope of $\cA$ also a necessary condition?

(c) Does the quotient property hold in any abelian envelope for simple objects, i.e., is a simple object in the abelian envelope always a quotient of an object in the original category? Note that this is true in the example from \Cref{intro-a} by the proof of \Cref{prop:concentrated} together with \Cref{prop:decomp} and \Cref{cor:tilt-res}.
\end{question}

\subsection*{Use of AI} The first version of this proof was found by GPT-6 Astra on September~5, 2026. We simplified and checked the proof, and wrote it up by hand, before using GPT-6 Astra and Claude Opus to check our document.

\subsection*{Acknowledgments} J.F.~was supported by the Hausdorff Center for Mathematics Bonn. J.F.'s Claude license was provided by the 2024 Max Planck--Humboldt Research Award (Max Planck Society / Alexander von Humboldt Foundation) bestowed on Geordie William\-son and hosted by Catharina Stroppel. J.F.~and T.H.~thank the ARTIN network as some research for this article was carried out during the ARTIN 2026 meeting in Newcastle.

%%%%%%%%%%%%%%%%%%%%%%%%%%%%%%%%%%%%%%%%%%%%%%%%%%%%%%%%
\subsection*{Conventions}
We keep our conventions consistent with \cite{CEOP} and \cite{FG}.

Tensor categories, pseudo-tensor categories, and abelian envelopes are defined as in \cite{CEOP}: tensor categories are abelian linear monoidal rigid with one-dimensional endomorphism algebra of the tensor unit. Pseudo-tensor categories are not necessarily abelian, but additive and idempotent complete instead. An abelian envelope is a fully faithful embedding $\cA\subset\cB$ of a pseudo-tensor category into a tensor category such that restriction along this embedding induces an equivalence, for any tensor category $\cT$, between the exact linear monoidal functors of the form $\cB\to\cT$ and the faithful linear monoidal functors of the form $\cA\to\cT$.

For an object $X$, its left dual is denoted $X^*$ and the corresponding coevaluation map is of the form $\one\to X\o X^*$. Likewise $^*X$ denotes the right dual.

Complexes are cohomologically indexed as in \cite{FG}, with differentials $C_i\to C_{i+1}$ and shifts $C[1]_i=C_{i+1}$. $K^b(\cA)$ and $D^b(\cB)$ denote the bounded homotopy category of an additive category $\cA$ and the bounded derived category of an abelian category $\cB$, respectively.

We read diagrams from source = top to target = bottom.

Throughout, $\kk$ is a fixed algebraically closed field.\footnote{We expect all results to hold over any field. We cite results from the literature that formulate this assumption, but we believe that their results, as far as relevant for our purposes, do not depend on it.} 

\section{Continuants and abelian envelopes} \label{sec:continuants}

Let $\cA$ be an additive rigid monoidal category. Our main tool in this section will be the \emph{continuants} $\EE_n$ considered in \cite{CE}. These complexes are built entirely from tensor words in an object and its iterated duals, with differentials obtained from evaluation. They can therefore be transported termwise by a monoidal functor defined only on the additive rigid monoidal category $\cA$. We will recall them using some combinatorial definitions that we will reuse later.

\begin{definition} \label{def:admissible}
    For any finite sequence of integers $w=(w_1,\dots,w_r)$, we define a set $I_\cup(w)$ of possible cup positions and a set $I_\cap(w)$ of possible cap positions as
$$
I_\cup(w) := \{ j: 1\le j\le r-1, w_j = w_{j+1}+1 \} ,\quad
I_\cap(w) := \{ j: 1\le j\le r-1, w_{j+1} = w_j+1 \} .
$$
A subset of $I_\cup(w)$ or $I_\cap(w)$ is called \emph{admissible} if it contains no two consecutive integers. Let $I^a_\cup(w)$ and $I^a_\cap(w)$ be the sets of admissible subsets of $I_\cup(w)$ and $I_\cap(w)$, respectively. For any $S$ in $I^a_\cup(w)$ or $I^a_\cap(w)$, let $w\setminus S$ denote the sequence obtained by removing the entries at positions $S$ and their successors in $w$. 
\end{definition}

For any finite sequence $w=(w_1,\dots,w_r)$ of integers and $X\in\cA$, we set
$$
X^{*w} := X^{* w_1}\o\dots\o X^{* w_r} ,
$$
where $X^{*0}:=X$, $X^{*j+1}:=(X^{*j})^*$, and $X^{*-j-1}:={}^*(X^{*-j})$ for all $j\ge0$. Then, by \cite{CE}*{Remark~2.1.4}, the continuant $\EE_n(X)$ is the complex 
$$
(\EE_n(X))_q := \bigoplus_{S\in I^a_\cup((n-1,\dots,1,0)), |S|=q} X^{* (n-1,\dots,1,0)\setminus S}
\quad\tforall q\ge0 .
$$
Set $w:=(n-1,\dots,1,0)$. Then the non-zero components of the differential of the complex are given by the maps 
$$(-1)^{|\{j\in S:j<i\}|} \ev_{S,i}.$$
Here $\ev_{S,i}\:  X^{*w\setminus S}\to X^{*w\setminus(S\cup\{i\})}$ denotes the evaluation map applied at positions $i,i+1$ of $w$ in the tensor product, whenever $i\in I_\cup(w)\setminus S$ and $S\cup\{i\}$ is admissible. This is indeed a complex as evaluations on disjoint pairs of tensor factors commute.

For example, 
$$
\EE_0(X) = \one ,\quad
\EE_1(X) = X, \quad
\EE_2(X) = [X^* \o X \xrightarrow{\ev_X} \one] ,
$$
where the leftmost term in each complex is in degree $0$. 

We use the following results on continuants:

\begin{proposition} \label{prop:EC} Fix $X\in\cA$.
\begin{enumerate}
\renewcommand{\theenumi}{\alph{enumi}}
\item \cite{CE}*{Proposition~2.2.5(1)} For every $n\ge1$, there is an exact triangle in $K^b(\cA)$
$$
\EE_{n-1}(X)^* \o \EE_{n+1}(X) \to \EE_n(X)^*\o \EE_n(X) \xrightarrow{\ev_{\EE_n(X)}} \one \to
$$
\item \cite{CE}*{Theorem~4.2.3} If $\cA$ is a tensor category and $\EE_n(X)\neq0$ in $D^b(\cA)$ for all $n\ge0$, then $\EE_n(X)$ has cohomology concentrated in degree $0$, with $H^0\EE_n(X)\neq0$, for all $n\ge0$.
\end{enumerate}
\end{proposition}

\begin{proof} For (b), note that loc.cit.~strictly speaking only applies to finite length categories, but the proof there does not depend on this: assume $\EE_n(X)$ and $\EE_{n-1}(X)$ are quasi-isomorphic to $A[0]$ and $B[0]$, respectively, for some $n\ge1$ and some non-zero objects $A,B$. Then by (a), we have an exact triangle
$$
B^* \o \EE_{n+1}(X) \to A^*\o A[0] \xrightarrow{\ev_A} \one[0] \to .
$$
As $A\neq0$, the map $\ev_A$ is an epimorphism. Hence $B^*\o \EE_{n+1}(X)$ is quasi-isomorphic to the kernel of $\ev_A$ placed in degree $0$. Tensoring with $B^*$ is exact and faithful, which implies the assertion for $\EE_{n+1}(X)$.
\end{proof}

\medskip

The following is the main technical tool to prove the desired abelian envelopes:

\begin{proposition} \label{prop:concentrated}
Let $\cA$ be a full additive rigid monoidal subcategory of a tensor category $\cC$.
Assume $X\in\cA$ and $\EE_n(X)\neq0$ in $D^b(\cC)$ for all $n\ge0$.
Then for any faithful linear monoidal functor $F\:\cA\to\cT$ to a tensor category $\cT$ and for all $n \geq 0$, we have
$$
F(\EE_n(X)) \simeq H^0 F(\EE_n(X))[0]
$$
in $D^b(\cT)$ and $H^0 F(\EE_n(X)) \neq 0$.
\end{proposition}

\begin{proof} By \Cref{prop:EC}(b), it suffices to show that $H^0 F(\EE_n(X))\neq0$ for all $n\ge0$. 

For any $n\ge0$, set $K_n:=H^0\EE_n(X)$ and let $d_n\:\EE_n(X)_0\to\EE_n(X)_1$ be the first differential of the continuant. Thus $K_n\cong\ker d_n$, and by \Cref{prop:EC}(b), we have $\EE_n(X)\simeq K_n[0]$ in $D^b(\cC)$ and $K_n\neq0$.

We claim that it suffices to show that for any $n\ge0$, there is an object $T_n\in\cA$ and an epimorphism $T_n\tto K_n$.
Indeed, given such an epimorphism, set $g_n:= T_n\tto K_n\hookrightarrow \EE_n(X)_0$. Then $d_n g_n=0$ and the map $T_n\to K_n$ induced by $g_n$ is an epimorphism. As $K_n\neq0$, the morphism $g_n$ is non-zero and belongs to $\cA$ by fullness. Hence, for faithful $F$, $F(g_n)$ induces a non-zero morphism to $H^0 F(\EE_n(X))$, showing it is non-zero, as required.

To construct $T_n$ as required, let $\cQ$ be the class of quotient objects of objects from $\cA$ in $\cC$. $\cQ$ is closed under quotients, tensor products, and double duals. We claim that it is also closed under forming kernels of non-zero morphisms to $\one$: assume there are epimorphisms $p\:T\tto Y$ and $u\:Y\tto\one$, where $T\in\cA$ and $Y\in\cC$. Then 
$$
T\o T \xrightarrow{up\o\id-\id\o up} T\xrightarrow{up} \one
$$
is exact at $T$ \cite{CEOP}*{Proposition~2.3.2}, and $p$ induces an epimorphism $\ker(up)\tto\ker(u)$. Thus $\ker(u)$ is a quotient of $T\o T$, and lies in $\cQ$.

We have to show that $\cQ$ contains $K_n$ for all $n\ge0$. This is clear for $n=0,1$ by definition. Assume it holds for $n-1$ and $n$, for some $n\ge1$. Dualizing the inclusion $K_n\subset \EE_n(X)_0\in\cA$ shows $K_n^*\in\cQ$. So by the above closure properties, the kernel of $\ev_{K_n}$ lies in $\cQ$. By \Cref{prop:EC}(a), this is $K_{n-1}^*\o K_{n+1}$. As $K_{n-1}\neq0$, we obtain an epimorphism
$$
K_{n-1}^{**} \o K_{n-1}^* \o K_{n+1} \xrightarrow{\ev_{K_{n-1}^*}\o\id} K_{n+1} ,
$$
showing $K_{n+1}\in\cQ$, since $K_{n-1}^{**}$ and $K_{n-1}^*\o K_{n+1}$ are already in $\cQ$. This completes the induction.
\end{proof}

Using the preceding proposition, we can give a sufficient criterion for when a tensor category with a lower finite highest weight structure is the abelian envelope of its subcategory of tilting objects, assuming that the latter is a pseudo-tensor subcategory.
Recall that a lower finite highest weight category is (loosely speaking) a finite length linear abelian category $\cC$ with finite-dimensional $\Hom$-spaces and with classes of \emph{standard objects} $\Delta(\lambda)$ and \emph{costandard objects} $\nabla(\lambda)$, indexed by a lower finite poset $(\Lambda,\leq)$ and satisfying certain homological properties; we refer the reader to \cite{BS} for additional details.
An object of $\cC$ is called a \emph{tilting object} if it admits a standard filtration and a costandard filtration, and we write $\Tilt\cC$ for the full subcategory of tilting objects in $\cC$.
A key property of tilting objects is that the canonical functor
\begin{equation}
\label{eq:tilting-equiv}
    K^b( \Tilt\cC ) \longrightarrow D^b( \cC )
\end{equation}
from the bounded homotopy category of $\Tilt\cC$ to the bounded derived category of $\cC$ is an equivalence of triangulated categories.

\begin{theorem} \label{thm:main-general} Let $\cB$ be a lower finite highest weight category that is also a tensor category. Assume its full subcategory $\cA$ of tilting objects is a pseudo-tensor subcategory. Assume there is a family $(Y_i)_{i\in I}$ of objects in $\cA$ such that $\EE_n(Y_i)\neq0$ in $D^b(\cB)$ for all $n\ge0$ and all $i\in I$. Assume all standard objects and all costandard objects in $\cB$ are contained in the full subcategory $\cB_0$ of $\cB$ generated by $\cA$ and all $K_{i,n}:=H^0\EE_n(Y_i)\in\cB$ under isomorphisms, tensor products, finite direct sums, direct summands, left and right duals, and extensions. Then the inclusion $\cA\subset\cB$ is an abelian envelope.
\end{theorem}

\begin{proof}
Let $F\:\cA\to\cT$ be a faithful linear monoidal functor to a tensor category.
Using the equivalence \eqref{eq:tilting-equiv} and applying $F$ termwise yields a monoidal triangulated functor
$$
\tilde F: D^b(\cB) \simeq K^b(\cA) \longrightarrow D^b(\cT) .
$$
By \Cref{prop:concentrated}, we have
$$
\tilde F(K_{i,n})\simeq F(\EE_n(Y_i))\simeq H^0 F(\EE_n(Y_i))[0] \in\cT[0].
$$
Hence, the full subcategory of objects $Y\in\cB$ such that $\tilde F(Y)\in\cT[0]$ contains all tilting objects and all $K_{i,n}$. It is closed under isomorphisms, tensor products, finite direct sums and summands, left and right duals, and extensions. Thus it contains all standard and costandard objects by our assumptions. By \cite{FG}*{Lemma~2.22(2)}, the standard objects with non-negative shifts generate $D^b(\cB)^{\le0}$ under extensions, and the costandard objects with non-positive shifts generate $D^b(\cB)^{\ge0}$. Hence, $\tilde F$ is $t$-exact. So $\tilde F$ restricts to an exact linear monoidal functor $\overline{F}\:\cB\to\cT$ such that $F\cong\overline{F} \circ (\cA\hookrightarrow\cB)$.

Now assume $G_1,G_2\:\cB\to\cT$ are exact linear monoidal functors and $\alpha\:G_1|_{\cA}\xrightarrow{\cong} G_2|_{\cA}$ is a monoidal natural isomorphism. Applying $\alpha$ termwise, using the equivalence \eqref{eq:tilting-equiv} again, and restricting to the hearts produces a natural monoidal isomorphism $\beta\:G_1\xrightarrow{\cong} G_2$ that restricts to $\alpha$ on $\cA$. For any $M\in\cB$, choose a bounded tilting complex representing $M$. Any extension of $\alpha$ between the exact functors $G_1,G_2$ induces on this complex the termwise map given by $\alpha$. Naturality with respect to its isomorphism with $M$ in $D^b(\cB)$ determines $\beta_M$ uniquely. Monoidal natural transformations are invertible by rigidity, and exact linear monoidal functors between tensor categories are faithful. This proves full faithfulness of restriction.
\end{proof}

\section{The universal pseudo-tensor category}
The main players of this section are the concrete categories $\cP$ and $\cA$ defined as follows:

Let $\cP$ be the free rigid monoidal $\kk$-linear category generated by an object $X_0$, as described in \cite{CEOP}*{Example~1.1.2} or \cite{CSV}. It contains objects $(X_i)_{i\in\ZZ}$ such that $X_{i+1}$ is the left dual and $X_{i-1}$ is the right dual of $X_i$, for all $i$. Let $\Lambda$ be the set of finite, possibly empty integer sequences $w=(w_1,\dots,w_r)$, then the objects of $\cP$ are given by the tensor products 
$$
X_w:=X_0^{* w} = X_{w_1}\o\dots\o X_{w_r}
\quad\tforall w\in\Lambda .
$$
Its morphisms are generated by co-/evaluation morphisms. In other words, let $P_{w,w'}$ be the set of planar / non-crossing arc diagrams whose upper and lower points are labeled by the sequences $w$ and $w'$ from $\Lambda$, respectively, subject to the condition that every (upper) cup has labels $(i+1,i)$, every (lower) cap has labels $(i,i+1)$, and every through-strand has labels $(i,i)$, for some $i$. Then 
$$
\cP(X_w,X_{w'}) = \kk P_{w,w'} .
$$
Identities in $\cP$ are given by diagrams without cups or caps.  Composition is given by the usual stacking of diagrams; note that no loops can arise in this way. Finally, tensor products are formed by juxtaposing diagrams and concatenating integer sequences. In particular, the tensor unit is given by the empty integer sequence.

Let $\cA$ be the additive and idempotent completion of $\cP$. It is a pseudo-tensor category over $\kk$ (in the sense of \cite{CEOP}), and can be regarded as the free pseudo-tensor category generated by an object $X_0$.

We define an additional structure on the category $\cP$. 

\begin{definition}
    Let $\cU$ and $\cD$ be the wide $\kk$-linear subcategories of $\cP$ whose morphisms are linear combinations of those diagrams from $\bigsqcup_{w,w'\in\Lambda} P_{w,w'}$ that have no cups or no caps, respectively. 
\end{definition}

In other words, $\cU$ is generated under composition and tensor product by diagrams whose components have at most one upper point.

For integer sequences $w,w'$, we denote by $w+w'$ their concatenation.

\begin{definition} Let $\le$ be the minimal reflexive and transitive relation on $\Lambda$ defined by setting
$$
w+w' \le w+(i,i+1)+w' ,
\quad
w+w' 
\le w+(i,i-1)+w'
\quad 
\tforall 
w,w'\in\Lambda, i\in\ZZ  .
$$
\end{definition}

Two sequences $w,w'$ are comparable in this order if one can be turned into the other by a sequence of removals of adjacent pairs of consecutive integers. 

Note that $w\le w'$ implies that the length of $w$ is at most the length of $w'$, and the length can be equal only if $w=w'$; in particular, $\le$ is antisymmetric and hence a partial order.
Distinct words give non-isomorphic objects, as the number of through-strands cannot increase under composition, and an all-through diagram between words of equal length exists only if the words agree. Thus $\Lambda$ labels the isomorphism classes of objects in $\cP$.

\begin{lemma} \label{lem-1} The data $(\cU,\cD,\le)$ define a monoidal triangular structure on $\cP$ in the sense of \cite{SS}*{Definition~4.1, Section~4.11}, and both left and right duality in $\cP$ define an order-preserving monoidal duality in the sense of \cite{FG}*{Definition~4.7}.
\end{lemma}

\begin{proof} Endomorphism diagrams in $\cU$ or $\cD$ have to be identities, so $\cU(X,X)\cong\cD(X,X)\cong\kk$ for all $X\in\cP$. $\cU(X_w,X_{w'})$ or $\cD(X_w,X_{w'})$ can only be non-zero if $w\le w'$ or $w'\le w$, respectively, and the poset $(\Lambda,\le)$ is lower finite by the arguments above. Finally, the linear maps
$$
\bigoplus_{Y\in\cP}
\cU(Y,Z) \o_{\cU(Y,Y)} \cD(X,Y) \to \cP(X,Z)
$$
constructed using composition in $\cP$ are isomorphisms for all $X,Z\in\cP$, as $\cU(Y,Y)=\kk \id_Y$ for any $Y\in\cP$ and any basis diagram in $\cP(X,Z)$ is uniquely a composition of a pair of diagrams without cups or caps, respectively. This verifies the axioms of a triangular structure.

$\cU$ and $\cD$ are closed under tensor products, verifying the axioms of a monoidal triangular structure.

The left dual of an object $X_w$ for $w=(w_1,\dots,w_r)$ is given by $X_{w^*}$ with
\begin{equation}
w^*=(w_r+1,w_{r-1}+1,\dots,w_1+1) ,
\label{eq:dual-words}    
\end{equation}
and similarly for right duals. This implies the axioms of an order-preserving monoidal duality.
\end{proof}

For a small $\kk$-linear category $\cC$, we denote its category of presheaves $\Fun_\kk(\cC\op,\Vect_\kk)$ valued in (possibly infinite-dimensional) vector spaces by $\PSh(\cC)$. $\PSh(\cC)$ is naturally abelian, and if $\cC$ is monoidal, then Day convolution defines a bifunctor on $\PSh(\cC)$ that is generally right exact in each argument.

\begin{lemma} \label{lem-2} The bifunctor on $\PSh(\cU)$ defined by Day convolution is exact in each argument.
\end{lemma}

\begin{proof} For $u,v,w\in\Lambda$, we define a set of \emph{reduced} diagrams\footnote{The idea of using reduced diagrams in this sense is very similar in spirit to \cite{SS}*{e.g., Proof of Proposition~5.11} or the more abstract version in \cite{FG}*{Appendix A}.} $R_{u,v,w}$ as those diagrams in $P_{u,v+w}$ that lie in $\cU$ and in which every cap connects a point in $v$ with a point in $w$. We claim that for any $f\in P_{u,v+w}$ lying in $\cU$, there exist unique $a,b\in\Lambda$, $f_0\in R_{u,a,b}$, $\alpha\in P_{a,v}$, $\beta\in P_{b,w}$ such that $\alpha,\beta\in\cU$ and
$$
f = (\alpha\o\beta) f_0
\quad\text{in }\cU .
$$
Indeed, $\alpha$ and $\beta$ are uniquely defined by those caps in $f$ that connect points within $v$ or $w$, respectively, and $f_0$ is uniquely defined by the remaining caps in $f$.

The claim implies that
$$
\cU(X_u, X_v\o X_w) \cong
\bigoplus_{a,b\in\Lambda,f_0\in R_{u,a,b}} \cU(X_a,X_v)\o_\kk \cU(X_b,X_w) ,
$$
which implies that for any $M,N\in\PSh(\cU)$ and any $X_u\in\cP$,
$$
(M\o N)(X_u) \cong \bigoplus_{a,b\in\Lambda,f_0\in R_{u,a,b}} M(X_a)\o_\kk N(X_b) . 
$$
Indeed, this can be read as an isomorphism of functors, and fixing $M$ or $N$ yields a formula that implies exactness of the resulting functor.
\end{proof}

For a small $\kk$-linear category $\cC$, let $\PSh\fd(\cC):=\Fun_\kk(\cC\op,\Vect\fd)$ be the category of pointwise finite-dimensional presheaves. Set $\cR:=\PSh\fd(\cP\op)$.

Recall that $\cA$ is the additive and idempotent completion of $\cP$.

\begin{theorem} \label{thm-1} The monoidal Ringel dual $\cB:={}^\vee(\cR\op)$ in the sense of \cite{FG} is a lower finite highest weight category with weight set $\Lambda$ whose (full) subcategory of tilting objects is monoidally equivalent with $\cA$. $\cB$ is a tensor category in the sense of \cite{CEOP}, in particular, rigid monoidal with biexact tensor product, and $\cA$ viewed as a subcategory is closed under taking duals.
\end{theorem}

\begin{proof} First note that $\cD\op\simeq\cU^{\text{rev}}$ by the duality in $\cP$. Now the assertions follow directly from \cite{FG}*{Theorem~4.9}, where \Cref{lem-1} and \Cref{lem-2} verify the relevant assumptions. To see that $\cB$ is a tensor category, we additionally observe that the tensor unit has a one-dimensional endomorphism algebra in $\cP$.

To identify the entire weight set, let $\cM:=\cU\cap\cD$. Then $\cM(X_v,X_w)=0$ for $v\ne w$ and $\cM(X_w,X_w)=\kk$. Thus the simple functors on $\cM$ are exactly the one-dimensional functors supported at a single sequence $w$. By \cite{FG}*{Proposition~4.3, Proof of Theorem~4.9}, these label all simple objects of $\cR$ and $\cR\op$, and Ringel duality retains the same label set. Hence $\Lambda$ is the entire weight set of $\cB$.
\end{proof}

\medskip

By \Cref{thm-1} and the equivalence in \eqref{eq:tilting-equiv}, the canonical functor 
\[
K^b(\cA) \longrightarrow D^b(\cB)
\]
from the bounded homotopy category of the additive category $\cA$ to the bounded derived category of the abelian category $\cB$
is an equivalence of triangulated categories.
Let $T(w)$ denote the image of $X_w$ in $\cB$; this is a tilting object which need not be indecomposable.

\medskip

For any $w\in\Lambda$, we denote the co-/standard objects in $\cB$ by $\nabla(w)$ and $\Delta(w)$, respectively. Our next goal is to determine co-/presentations and, eventually, tilting co-/resolutions for these 
objects.

Recall \Cref{def:admissible}.

For $j\in I_\cup(w)$, let $\ev_j\:T(w)\to T(w\setminus\{j\})$ be the diagram with exactly one cup at position $j$; similarly, for any $j\in I_\cap(w)$, let $\co_j\: T(w\setminus\{j\})\to T(w)$ be the diagram with one cap at position $j$. We can think of these as morphisms in $\cB$.

\begin{lemma} \label{lem:presentations} The co-/standard objects have the following co-/presentations by tilting objects in $\cB$:
$$
\nabla(w) \cong T(w) \big/  \sum_{j\in I_\cap(w)} \im \co_j ,\quad
\Delta(w) = \bigcap_{j\in I_\cup(w)} \ker \ev_j \subset T(w) .
$$
\end{lemma}

\begin{proof}
Using the bijection between standard and costandard objects in $\cB$ induced by duality in $\cP$ (see \cite{FG}*{Theorem~2.25}), each statement implies the other.

By Ringel duality (\Cref{thm-1}), showing the asserted presentation of $\nabla(w)$, 
\begin{equation}
\bigoplus_{j\in I_\cap(w)} T(w\setminus\{j\}) 
\xrightarrow{(\co_j)_j} 
T(w) 
\tto \nabla(w)
\label{eq:pres}
\end{equation}
follows by showing that replacing tilting objects by representable functors in $\cR\op\simeq\PSh\fd(\cP)$ gives a projective presentation of the standard object $\Delta_{\cR\op}(w)$ in $\cR\op$. Here the equivalence with presheaves is given by pointwise vector-space duality. More precisely, the opposite Ringel duality functor on cocompletions is right exact, sends the representable $\cP(-,X_v)$ to $T(v)$, and sends $\Delta_{\cR\op}(w)$ to $\nabla(w)$; see \cite{FG}*{Theorem~2.25 and its proof}. Applying this functor to the projective presentation therefore gives \Cref{eq:pres}.

Consider the inclusion $\iota\:\cU\to\cP$. As explained in \cite{SS}*{Section~3.6, Propositions~4.4 and 4.5}, the induced functor $\PSh\fd(\cP)\to\PSh\fd(\cU)$ has a left adjoint $\iota_!\:\PSh\fd(\cU)\to\PSh\fd(\cP)$ that is exact and sends projectives to projectives and simple objects to standard objects.

Hence, it suffices to show that replacing tilting objects in \Cref{eq:pres} by representable functors in $\PSh\fd(\cU)$ gives a presentation of $s_w$, the simple object given by the one-dimensional presheaf supported at $X_w$.

Fix $v\in\Lambda$. Any diagram in $P_{v,w}$ that lies in $\cU$ but is not the identity on $w$ has caps, at least one of which must be connecting adjacent lower points. Hence, it factors via a morphism of the form $\co_j$. Conversely, the images of the maps $\co_j$ contain no non-zero multiple of $\id_{X_w}$. This implies the asserted presentation of $s_w$. 
\end{proof}

The above co-/presentations imply tensor product decompositions and some general consequences on the highest weight category $\cB$. For $s\in\ZZ$ and $n\ge0$, we define ascending and descending sequences 
$$
w^+_{s,n}:=(s,s+1,\dots,s+n-1) ,\quad
w^-_{s,n}:=(s+n-1,\dots,s+1,s)
\quad\in\Lambda .
$$
In particular, the sequences are empty if $n=0$. Note that 
$$
(w^\pm_{s,n})^*=w^\mp_{s+1,n} ,\quad (w^\pm_{s,n})^{**}=w^\pm_{s+2,n} , $$
(see \Cref{eq:dual-words}).

Any sequence $w\in\Lambda$ has a unique decomposition into a minimal number of ascending sequences, which we call the \emph{ascending blocks} of $w$. Similarly, the \emph{descending blocks} of $w$ are defined as the minimal decomposition into descending sequences.

\begin{proposition} \label{prop:decomp}
For any $w\in\Lambda$ with a decomposition into ascending/descending blocks
$$
w = w^+_{s_1,n_1} + \dots + w^+_{s_p,n_p}
= w^-_{t_1,m_1} + \dots + w^-_{t_q,m_q} ,
$$
there are isomorphisms
$$
\Delta(w) \cong \Delta(w^-_{t_1,m_1}) \o\dots\o \Delta(w^-_{t_q,m_q}) , \quad
\nabla(w) \cong \nabla(w^+_{s_1,n_1}) \o\dots\o \nabla(w^+_{s_p,n_p}) .
$$
Moreover, for all $s\in\ZZ$ and $n\ge0$, the standard objects $\Delta(w^-_{s,n})$ are simple, and we have
$$
\nabla(w^-_{s,n}) \cong T(w^-_{s,n}) ,\quad \Delta(w^+_{s,n}) \cong T(w^+_{s,n}) ,
$$
$$
\nabla(w^\pm_{s,n}) \cong \Delta(w^\mp_{s-1,n})^* .
$$
\end{proposition}

\begin{proof} In any tensor category, we have the identities
$$
(\ker f)\o(\ker g) = \ker (f\o\id) \cap \ker (\id\o g) ,\quad
(X / X')\o (Y / Y') =
\frac{X\o Y}{X'\o Y + X\o Y'} ,
$$
as subobject and quotient, respectively, for all morphisms $f,g$ and all subobjects $X'\subset X$ and $Y'\subset Y$.

This applied to the co-/presentations from \Cref{lem:presentations} together with the observation that for all non-empty $w=(w_1,\dots,w_r),w'=(w'_1,\dots,w'_{r'})\in\Lambda$ such that $w_r\neq w'_1+1$ or $w_r\neq w'_1-1$, respectively, the sets $I_\cup(w+w')$ and $I_\cap(w+w')$ are a disjoint union of the respective sets for $w$ and $w'$, up to relabeling,
 implies the asserted tensor product decompositions.

$\nabla(w^-_{s,n})$ and $\Delta(w^+_{s,n})$ are isomorphic to the corresponding tilting object as $I_\cap(w^-_{s,n})$ and $I_\cup(w^+_{s,n})$ are empty sets.
This implies that the canonical map $$
\Delta(w^-_{s,n})\hookrightarrow T(w^-_{s,n})\xrightarrow{\cong} \nabla(w^-_{s,n})
$$ is a monomorphism. As its image is the simple object corresponding to $w^-_{s,n}$, $\Delta(w^-_{s,n})$ is simple.
The remaining isomorphisms $\nabla(w^\pm_{s,n})\cong \Delta(w^\mp_{s-1,n})^*$ are again by duality.
\end{proof}

We now connect these concrete categories with the general theory from \Cref{sec:continuants}.

\begin{corollary} \label{cor:tilt-res} For any $s\in\ZZ$ and $n\ge0$, the complex $\EE_n(X_s)$ defines a tilting coresolution of $\Delta(w^-_{s,n})$, that is, there is an isomorphism
$$
    \EE_n(X_s) \cong \Delta(w^-_{s,n})[0]
$$
in $D^b(\cB)$.
\end{corollary}

\begin{proof}
    The cohomology of $\EE_n(X_s)$ in degree $0$ is $\Delta(w^-_{s,n})\neq0$ by \Cref{lem:presentations}. Hence, \Cref{prop:EC}(b) implies that, in fact, the cohomology of $\EE_n(X_s)$ is concentrated in degree $0$. This shows the assertion.
\end{proof}

\begin{remark} For arbitrary $w\in\Lambda$ and $X\in\cA$, define
$$
(\EE_w(X))_q := \bigoplus_{S\in I^a_\cup(w), |S|=q} X^{* w\setminus S}
\quad\tforall q\ge0 ,
$$
with differentials given by the maps 
$$(-1)^{|\{j\in S:j<i\}|} \ev_{S,i} \: X^{* w\setminus S} \to X^{*w\setminus(S\cup\{i\})},
$$
for all $S,i$ such that $i\notin S$ and $S,S\cup\{i\}$ are in $I^a_\cup(w)$.
Then \Cref{cor:tilt-res} together with the tensor product formulas in \Cref{prop:decomp} implies that $\EE_w(X_0)$ is a tilting coresolution of $\Delta(w)$.
\end{remark}

\begin{corollary} \label{cor:ab-env} The embedding $\cA\subset\cB$ is an abelian envelope.
\end{corollary}

\begin{proof} This now follows from \Cref{thm:main-general}, where \Cref{prop:decomp} and \Cref{cor:tilt-res} verify the assumptions.
\end{proof}

\Cref{cor:ab-env} and the arguments leading to its proof imply:

\begin{corollary} \label{rem:mapping-prop} For any tensor category $\cT$, the following categories are equivalent:
\begin{enumerate}
\renewcommand{\theenumi}{\alph{enumi}}
\item the category of tensor functors (i.e., exact linear monoidal functors) of the form $\cB\to\cT$,
\item the category of faithful linear monoidal functors of the form $\cA\to\cT$, and 
\item the groupoid of \emph{unbounded} objects in $\cT$ in the sense of \cite{CE}, i.e., objects $X$ such that $H^0\EE_n(X)\neq0$ for all $n\ge0$.
\end{enumerate}
\end{corollary}

\begin{proof}
The categories in (a) and (b) are equivalent because $\cB$ is an abelian envelope of $\cA$, and for any faithful linear monoidal functor $F\: \cA \to \cT$, the object $F(X_0)$ is unbounded by \Cref{prop:concentrated} and \Cref{cor:tilt-res}.
In order to show that the category in (c) is equivalent to those in (a) and (b), it now remains to observe that for an unbounded object $X$ of $\cT$, the unique (up to isomorphism) linear monoidal functor $F_X\:\cA\to\cT$ with $F_X(X_0) \cong X$ is faithful.
Note that unboundedness is preserved by taking double duals.
Using the triangles
$$
\EE_{n+1}(X)\to\EE_n(X^*)\o X\to\EE_{n-1}(X^{**})\to
\quad(n\ge1)
$$
from \cite{CE}*{Proposition~2.2.4}, \Cref{prop:EC}(b), and exactness and faithfulness of $-\o X$, we see that unboundedness is also preserved by taking left duals, and analogously, by taking right duals.
Now using the equivalence \eqref{eq:tilting-equiv} and applying $F_X$ termwise gives rise to a monoidal triangulated functor
\[ \tilde F_X \colon D^b(\cB) \simeq K^b(\cA) \longrightarrow D^b(\cT) , \]
which sends all co-/standard objects to degree $0$ by \Cref{prop:decomp} and \Cref{cor:tilt-res}.
The $t$-exactness argument in the proof of \Cref{thm:main-general} implies that $\tilde F_X$ restricts to a tensor functor $\cB \to \cT$ extending $F_X \colon \cA \to \cT$.
As all tensor functors are faithful, we conclude that $F_X$ is faithful.
\end{proof}

\begin{remark} \label{rem:quantum-group} It can be seen that $\cB$ can be viewed as the category of finite-dimensional comodules of the free Hopf algebra $H$ with invertible antipode on the $2\times2$-matrix coalgebra, as discussed in \cites{Sch,Ch,RVdB}. 

Set $V:=\kk^2$ and $C:=V^*\o_\kk V$, the matrix coalgebra. $V$ is an unbounded object in the finite length category $\Vect\fd$ by \cite{CE}*{Theorem~4.3.1}, as it is not simple. Hence, by \Cref{rem:mapping-prop}, it defines a tensor functor $\omega\:\cB\to\Vect\fd$ such that $\omega(X_0)=V$. Now by reconstruction there is a Hopf algebra $L$ with bijective antipode and a monoidal equivalence $\cB\simeq\Comod\fd(L)$ compatible with $\omega$. It remains to show $L\cong H$.

For any Hopf algebra $K$ with bijective antipode, a coalgebra map $C\to K$ is the same as a right $K$-coaction on the fixed vector space $V$. Every such comodule is unbounded, as the forgetful functor $U_K\:\Comod\fd(K)\to\Vect\fd$ is exact and faithful. By \Cref{rem:mapping-prop}, the coalgebra map therefore determines an exact linear monoidal functor $G\:\cB\to\Comod\fd(K)$ together with a monoidal isomorphism $U_KG\cong\omega$, unique up to unique compatible monoidal isomorphism once the identification on $X_0$ is fixed. By reconstruction, these data are equivalent to a Hopf algebra map $L\to K$. Hence, we get a natural isomorphism
$$
\Hom_{\mathrm{Hopf}}(L,K)
\cong \Hom_{\mathrm{Coalg}}(C,K),
$$
which is the defining universal property of $H$, proving $L\cong H$, as asserted.
\end{remark}

\newcommand\Rep{\operatorname{Rep}}

\begin{remark} To obtain more examples of abelian envelopes without the quotient property, one could, for instance, consider any finite group $G$ whose order is not divisible by the characteristic of $\kk$. Then the Deligne tensor product (as discussed in \cite{CF}*{Section~3}, for example) $\Rep G\boxtimes\cB$ is a tensor category equivalent to $\cB^{\oplus m}$ as an abelian category, for $m$ the number of isomorphism classes of irreducible representations of $G$ over $\kk$, and is the abelian envelope of the pseudo-tensor category $\Rep G\boxtimes\cA$.
\end{remark}

%%%%%%%%%%%%%%%%%%%%%%%%%%%%%%%%%%%%%%%%%%%%%%%%%%%%%%%%%%%%%%%%%%%%%

\bibliography{bib}{}

@article {CEOP,
    AUTHOR = {Coulembier, Kevin and Etingof, Pavel and Ostrik, Victor and
              Pauwels, Bregje},
     TITLE = {Monoidal abelian envelopes with a quotient property},
   JOURNAL = {J. Reine Angew. Math.},
  FJOURNAL = {Journal f\"ur die Reine und Angewandte Mathematik. [Crelle's
              Journal]},
    VOLUME = {794},
      YEAR = {2023},
     PAGES = {179--214},
      ISSN = {0075-4102,1435-5345},
   MRCLASS = {18M05 (14L17)},
  MRNUMBER = {4529414},
MRREVIEWER = {Johannes\ Flake},
       DOI = {10.1515/crelle-2022-0076},
       URL = {https://doi.org/10.1515/crelle-2022-0076},
}

@article{FG,
      title={Monoidal Ringel duality and monoidal highest weight envelopes}, 
      author={Johannes Flake and Jonathan Gruber},
      year={2026},
      eprint={arXiv:2512.19558},
      archivePrefix={arXiv},
      primaryClass={math.RT},
      url={https://arxiv.org/abs/2512.19558}, 
}

@article {SS,
    AUTHOR = {Sam, Steven V. and Snowden, Andrew},
     TITLE = {The representation theory of {B}rauer categories {I}:
              {T}riangular categories},
   JOURNAL = {Appl. Categ. Structures},
  FJOURNAL = {Applied Categorical Structures. A Journal Devoted to
              Applications of Categorical Methods in Algebra, Analysis,
              Order, Topology and Computer Science},
    VOLUME = {30},
      YEAR = {2022},
    NUMBER = {6},
     PAGES = {1203--1256},
      ISSN = {0927-2852,1572-9095},
   MRCLASS = {18G99 (16G10 17B10 17B20 20G05)},
  MRNUMBER = {4519632},
MRREVIEWER = {Mee\ Seong\ Im},
       DOI = {10.1007/s10485-022-09689-7},
       URL = {https://doi.org/10.1007/s10485-022-09689-7},
}

@article {CE,
    AUTHOR = {Coulembier, Kevin and Etingof, Pavel},
     TITLE = {{$N$}-spherical functors and tensor categories},
   JOURNAL = {Int. Math. Res. Not. IMRN},
  FJOURNAL = {International Mathematics Research Notices. IMRN},
      YEAR = {2024},
    NUMBER = {14},
     PAGES = {10615--10649},
      ISSN = {1073-7928,1687-0247},
   MRCLASS = {18M05 (18N60)},
  MRNUMBER = {4775671},
       DOI = {10.1093/imrn/rnae093},
       URL = {https://doi.org/10.1093/imrn/rnae093},
}

@article {Ch,
    AUTHOR = {Chirvasitu, Alexandru},
     TITLE = {Grothendieck rings of universal quantum groups},
   JOURNAL = {J. Algebra},
  FJOURNAL = {Journal of Algebra},
    VOLUME = {349},
      YEAR = {2012},
     PAGES = {80--97},
      ISSN = {0021-8693,1090-266X},
   MRCLASS = {16T05 (16E20)},
  MRNUMBER = {2853627},
MRREVIEWER = {Zhihua\ Wang},
       DOI = {10.1016/j.jalgebra.2011.09.020},
       URL = {https://doi.org/10.1016/j.jalgebra.2011.09.020},
}

@article {RVdB,
    AUTHOR = {Raedschelders, Theo and Van den Bergh, Michel},
     TITLE = {The {M}anin {H}opf algebra of a {K}oszul {A}rtin-{S}chelter
              regular algebra is quasi-hereditary},
   JOURNAL = {Adv. Math.},
  FJOURNAL = {Advances in Mathematics},
    VOLUME = {305},
      YEAR = {2017},
     PAGES = {601--660},
      ISSN = {0001-8708,1090-2082},
   MRCLASS = {16S38 (16E60 16S37 16T20 18D10)},
  MRNUMBER = {3570144},
MRREVIEWER = {Ji-Wei\ He},
       DOI = {10.1016/j.aim.2016.09.017},
       URL = {https://doi.org/10.1016/j.aim.2016.09.017},
}

@article {BS,
    AUTHOR = {Brundan, Jonathan and Stroppel, Catharina},
     TITLE = {Semi-infinite highest weight categories},
   JOURNAL = {Mem. Amer. Math. Soc.},
  FJOURNAL = {Memoirs of the American Mathematical Society},
    VOLUME = {293},
      YEAR = {2024},
    NUMBER = {1459},
     PAGES = {vii+152},
      ISSN = {0065-9266,1947-6221},
      ISBN = {978-1-4704-6783-8; 978-1-4704-7716-5},
   MRCLASS = {16-02 (08C20 16G10 17B10 17B67 18E10 20G05 20G42)},
  MRNUMBER = {4684337},
       DOI = {10.1090/memo/1459},
       URL = {https://doi.org/10.1090/memo/1459},
}

@article {CSV,
    AUTHOR = {Coulembier, Kevin and Street, Ross and van den Bergh, Michel},
     TITLE = {Freely adjoining monoidal duals},
   JOURNAL = {Math. Structures Comput. Sci.},
  FJOURNAL = {Mathematical Structures in Computer Science. A Journal in the
              Applications of Categorical, Algebraic and Geometric Methods
              in Computer Science},
    VOLUME = {31},
      YEAR = {2021},
    NUMBER = {7},
     PAGES = {748--768},
      ISSN = {0960-1295,1469-8072},
   MRCLASS = {18M05},
  MRNUMBER = {4386829},
MRREVIEWER = {Eliezer\ Batista},
       DOI = {10.1017/S0960129520000274},
       URL = {https://doi.org/10.1017/S0960129520000274},
}

@incollection {Sch,
    AUTHOR = {Schauenburg, Peter},
     TITLE = {Faithful flatness over {H}opf subalgebras: counterexamples},
 BOOKTITLE = {Interactions between ring theory and representations of
              algebras ({M}urcia)},
    SERIES = {Lecture Notes in Pure and Appl. Math.},
    VOLUME = {210},
     PAGES = {331--344},
 PUBLISHER = {Dekker, New York},
      YEAR = {2000},
      ISBN = {0-8247-0367-7},
   MRCLASS = {16W30},
  MRNUMBER = {1761130},
MRREVIEWER = {Stefaan\ Caenepeel},
}

@article {BEO,
    AUTHOR = {Benson, Dave and Etingof, Pavel and Ostrik, Victor},
     TITLE = {New incompressible symmetric tensor categories in positive
              characteristic},
   JOURNAL = {Duke Math. J.},
  FJOURNAL = {Duke Mathematical Journal},
    VOLUME = {172},
      YEAR = {2023},
    NUMBER = {1},
     PAGES = {105--200},
      ISSN = {0012-7094,1547-7398},
   MRCLASS = {18M05},
  MRNUMBER = {4533718},
MRREVIEWER = {Xiao-Kan\ Guo},
       DOI = {10.1215/00127094-2022-0030},
       URL = {https://doi.org/10.1215/00127094-2022-0030},
}

@article {Cou-hom,
    AUTHOR = {Coulembier, Kevin},
     TITLE = {Homological kernels of monoidal functors},
   JOURNAL = {Selecta Math. (N.S.)},
  FJOURNAL = {Selecta Mathematica. New Series},
    VOLUME = {29},
      YEAR = {2023},
    NUMBER = {2},
     PAGES = {Paper No. 24, 46},
      ISSN = {1022-1824,1420-9020},
   MRCLASS = {18M15 (18D15 18F10 18N40)},
  MRNUMBER = {4549962},
       DOI = {10.1007/s00029-023-00829-y},
       URL = {https://doi.org/10.1007/s00029-023-00829-y},
}

@article {Cou,
    AUTHOR = {Coulembier, Kevin},
     TITLE = {Monoidal abelian envelopes},
   JOURNAL = {Compos. Math.},
  FJOURNAL = {Compositio Mathematica},
    VOLUME = {157},
      YEAR = {2021},
    NUMBER = {7},
     PAGES = {1584--1609},
      ISSN = {0010-437X,1570-5846},
   MRCLASS = {18D15 (18F10 18F20 18M05 20G05)},
  MRNUMBER = {4277111},
MRREVIEWER = {Johannes\ Flake},
       DOI = {10.1112/S0010437X21007399},
       URL = {https://doi.org/10.1112/S0010437X21007399},
}

@article {CEAH,
    AUTHOR = {Coulembier, Kevin and Entova-Aizenbud, Inna and Heidersdorf,
              Thorsten},
     TITLE = {Monoidal abelian envelopes and a conjecture of {B}enson and
              {E}tingof},
   JOURNAL = {Algebra Number Theory},
  FJOURNAL = {Algebra \& Number Theory},
    VOLUME = {16},
      YEAR = {2022},
    NUMBER = {9},
     PAGES = {2099--2117},
      ISSN = {1937-0652,1944-7833},
   MRCLASS = {18M05 (14L15 16D90)},
  MRNUMBER = {4523326},
MRREVIEWER = {Daniel\ Ion\ Bulacu},
       DOI = {10.2140/ant.2022.16.2099},
       URL = {https://doi.org/10.2140/ant.2022.16.2099},
}

@article{FLP-abenv,
      title={Monoidal adjunctions and abelian envelopes}, 
      author={Johannes Flake and Robert Laugwitz and Sebastian Posur},
      year={2026},
      eprint={2601.16092},
      archivePrefix={arXiv},
      primaryClass={math.RT},
      url={https://arxiv.org/abs/2601.16092}, 
}

@article {CF,
    AUTHOR = {Coulembier, Kevin and Flake, Johannes},
     TITLE = {Towards higher {F}robenius functors for symmetric tensor
              categories},
   JOURNAL = {Selecta Math. (N.S.)},
  FJOURNAL = {Selecta Mathematica. New Series},
    VOLUME = {32},
      YEAR = {2026},
    NUMBER = {5},
     PAGES = {Paper No. 98},
      ISSN = {1022-1824,1420-9020},
   MRCLASS = {18M05 (20B05 20C20 20G05)},
  MRNUMBER = {5134613},
       DOI = {10.1007/s00029-026-01204-3},
       URL = {https://doi.org/10.1007/s00029-026-01204-3},
}
\bibliographystyle{plain}

\end{document}